\documentclass[reqno,a4paper, 10pt]{amsart}
\usepackage{amsfonts}
\usepackage{amssymb, amsmath,amsthm,enumerate}
\usepackage[mathscr]{eucal}
\usepackage{url}

\theoremstyle{plain}
\newtheorem{thm}{Theorem}
\newtheorem{cor}[thm]{Corollary}
\newtheorem{obs}[thm]{Observation}
\newtheorem{lem}[thm]{Lemma}
\newtheorem{con}[thm]{Conjecture}

\newtheorem*{GNCT}{The Guthrie-Nymann Classification Theorem}
\newtheorem*{OP}{Open Problem}

\theoremstyle{definition}

\newcommand{\N}{\mathbb N}

\begin{document}
\title{The Lebesgue measure of boundaries of multigeometric Cantorvals}
\author[P. Nowakowski and F. Prus-Wi\'{s}niowski]{ Piotr Nowakowski and Franciszek Prus-Wi\'{s}niowski}

\newcommand{\eacr}{\newline\indent}

\address{Faculty of Mathematics and Computer Science
\\University of \L\'{o}d\'{z}
\\Banacha 22,
90-238 \L\'{o}d\'{z}
\\Poland
\\ORCID: 0000-0002-3655-4991}
\email{piotr.nowakowski@wmii.uni.lodz.pl}

\address{\llap{*\,}Franciszek Prus-Wi\'{s}niowski\\
Instytut Matematyki\\
Uniwersytet Szczeci\'{n}ski\\
ul. Wielkopolska 15\\
PL-70-453 Szczecin\\
Poland\\ ORCID 0000-0002-0275-6122}
\email{franciszek.prus-wisniowski@usz.edu.pl}


\date{\today}

\newcommand{\acr}{\newline\indent}

\begin{abstract}
We prove that the boundary of every multigeometric Cantorval is a null set, and extend this result to a larger class of standard achievable Cantorvals. In addition, we discuss the sets of uniqueness of achievement sets and show that they always belong to the Borel class $\mathcal{G}_\delta$.
\end{abstract}

\maketitle

\section{Introduction}

Given a sequence $(x_n)$ of real numbers, the set
$$
E\ =\ E(x_n)\ :=\ \left\{y\in\mathbb R:\ \exists\,A\subset\mathbb N \quad y=\sum_{n\in A}x_n\,\right\}
$$
is called the \textsl{achievement set} of the sequence $(x_n)$ (or the \textsl{set of subsums} of the series $\sum x_n$).
The study of achievement sets began over a century ago with two papers of Kakeya \cite{Kakeya}, \cite{Kakeya2}, who investigated the topological structure of achievement sets of positive sequences $(x_n)$ with $\sum x_n < \infty$. Sequences for which $x_n \not\to 0$ are generally not considered—only \cite{Jones} contains a handful of remarks and open problems in this direction. Throughout this note, we therefore assume $x_n \to 0$.

The definition of an achievement set extends naturally when $(x_n)$ are vectors in a Banach space. A particularly interesting case occurs when $(x_n)$ are $d$-dimensional vectors in $\mathbb R^d$. For instance, a classical result of Erdős and Straus states that the achievement set of the two-dimensional series $\sum_{n=1}^\infty(\tfrac1n,\,\tfrac1{n+1})$ has non-empty interior. Very recently, this result was generalized to three dimensions in  a beautifully written paper \cite{Kov25}, thereby answering affirmatively a question posed by Erdős and Graham \cite[p.~65]{EG}. Unfortunately, the theory of higher-dimensional achievement sets remains difficult and underdeveloped, mainly because the classification of possible topological types is still unresolved (cf. \cite{KN24} and references therein).
On the other hand, achievement sets of real sequences have been studied intensively in recent years, and even the most classical one-dimensional results have found applications in other areas of mathematics. For example, Kova\v{c} and Tao used basic Kakeya-type observations several times in their recent work (see \cite[p.~585]{KT}).

If $\sum x_n$ does not converge absolutely, then $E(x_n)$ is either the entire real line or a closed half-line \cite[Thm.~21.3]{BFPW1}, and thus has no real topological complexity. If $\sum x_n$ converges absolutely, then
\[
E(x_n)=E(|x_n|)+\sum_{x_n<0}x_n \quad \text{\cite[p. 348]{BFPW1}, \cite{H}}.
\]
Hence, the study of metric and topological properties of $E(x_n)$ reduces to the non-negative case. From now on, we assume $x_n \ge 0$ for all $n$, so that $\sum x_n$ is a convergent series of non-negative, monotone terms. Kakeya proved that in this case, $E(x_n)$ is always compact and perfect.

For such a series $\sum x_n$, define
$$
F_k\ =\ F_k(x_n)\ :=\ \left\{ y\in\mathbb R:\ \exists_{A\subset\{1,2,\ldots,k\}}\quad y=\sum_{n\in A}x_n \right\}, \qquad k\in\mathbb N,
$$
with $F_0 := \{0\}$. Similarly, let $E_k = E_k(x_n) := E\bigl((x_n)_{n>k}\bigr)$ be the achievement set of the $k$-th remainder $r_k = \sum_{n>k} x_n$ for $k\in\mathbb N_0$. In particular, $E = F_k + E_k$ for every $k\in\mathbb N_0$.

Given a closed set $A \subset \mathbb R$, the bounded components of $\mathbb R \setminus A$ are called \textsl{$A$-gaps}, and the non-trivial components of $A$ are called \textsl{$A$-intervals}. The cornerstone of one-dimensional achievement set theory is the Guthrie–Nymann Classification Theorem \cite{GN88} (see also \cite{NS}).

\begin{GNCT}
The set $E$ of all subsums of an absolutely convergent series is always of one of the following four types:
\begin{itemize}
\item[(i)] a finite set;
\item[(ii)] a multi-interval set;
\item[(iii)] a Cantor set;
\item[(iv)] a Cantorval.
\end{itemize}
\end{GNCT}

A \textsl{multi-interval set} is the union of finitely many closed bounded intervals. A \textsl{Cantor set} is a set homeomorphic to the classical ternary Cantor set.
\textsl{Cantorvals} are the most intricate achievement sets: they are compact sets $A \subset \mathbb R$ such that every endpoint of an $A$-gap is simultaneously the limit of a sequence of $A$-gaps and of a sequence of $A$-intervals. Cantorvals were introduced in \cite{MO} under the name $M$-Cantorvals. Mendes and Oliveira showed that all Cantorvals are mutually homeomorphic \cite[Appendix]{MO}.
Clearly, all Cantor sets are homeomorphic, while multi-interval sets are homeomorphic if and only if they have the same number of components. The same holds for finite sets. Thus, the notion of the \textsl{topological type} of an achievement set is well justified. A natural example of a Cantorval was given by Guthrie and Nymann \cite[p.~326]{GN88}: the \textsl{model Cantorval}, obtained by taking the ternary Cantor set $C$ and adjoining all the intervals removed in the odd-numbered steps of its standard construction.

A natural question is how to determine the topological type of $E(x_n)$ from the terms of $\sum x_n$. Kakeya already studied the relationship between the terms $x_n$ and the remainders $r_n = \sum_{i>n} x_i$, restricting to convergent series of positive terms (which, as noted, entails no loss of generality). The set
\[
K(x_n) := \{n \in \mathbb N : x_n > r_n \}
\]
is called the set of \textsl{Kakeya conditions}, while
\[
K^c(x_n) := \{n \in \mathbb N : x_n \le r_n \}
\]
is the set of \textsl{reversed Kakeya conditions}.
Every subset of $\mathbb N$ falls into one of three categories: finite sets, cofinite sets (with finite complement), and infinite sets with infinite complement. The first two categories already appeared in Kakeya’s early results relating the terms of a series to the topology of its achievement set.

\begin{thm}
\label{t1}
$E(x_n)$ is a multi-interval set if and only if $K(x_n)$ is finite.
\end{thm}

\begin{thm}
\label{t2}
If $\mathrm{card}\,K^c(x_n)<\infty$, then $E(x_n)$ is a Cantor set.
\end{thm}
Kakeya conjectured in 1914 that the implication can be reversed \cite{Kakeya2}.

Every achievement set can be expressed as
$$
E\ =\ \bigcap_{n=1}^\infty I_n \qquad \text{\cite[Fact 21.8]{BFPW1}},
$$
where each \textsl{$n$-th iteration} $I_n$ is a multi-interval set defined by
$$
I_n\ :=\ \bigcup_{f\in F_n}\bigl(f+[0,r_n]\bigr)\ =\ \bigcup_{f\in F_n}[f,f+r_n].
$$
It is not difficult to see that the topological types of $E$ and $E_k$ are always the same, independently of $k$.

\subsection*{Multigeometric sequences}

Multigeometric sequences (or series) form a class that is relatively easy to analyze from the viewpoint of topological classification of achievement sets. Informally, a \textsl{multigeometric sequence} is a monotone mixture of finitely many convergent geometric sequences with the same ratio.

Formally, let $m \in \mathbb N$, $q \in (0,1)$, and $k_1 \ge k_2 \ge \ldots \ge k_m > 0$. A multigeometric sequence is then given by $a_n = k_i q^j$, where $(j,i)$ is the unique pair with $j \in \mathbb N$ and $i \in \{1,\ldots,m\}$ such that $n = (j-1)m + i$. Multigeometric series are denoted by
\[
\sum(k_1,\ldots,k_m;\,q),
\]
and their achievement sets by $E(k_1,\ldots,k_m;\,q)$.

Historically, the first counterexamples to Kakeya’s 1914 hypothesis were constructed as achievement sets of multigeometric sequences (see \cite{WS}, without proof, and \cite{F}, with proof). Later, Jones described an infinite family of multigeometric sequences whose achievement sets are Cantorvals \cite[p.~515]{Jones}. In \cite{BFS}, the first sufficient condition for a multigeometric series to yield a Cantorval was established \cite[Thm.~2.1]{BFS}. Further deep results from \cite{BBFS} apply to certain multigeometric sequences, providing insight into how the topological type depends on the ratio $q$ for a fixed set of initial terms $k_1q, k_2q, \ldots, k_mq$ (see \cite[pp.~1025--1026]{BBFS}).

An important example is the Guthrie–Nymann Cantorval $E(3,2;\tfrac14)$ \cite{GN88}, \cite{BPW}, which is the achievement set of a multigeometric sequence as well.

\subsection*{Tight subsets and intervals in achievement sets}

Given $\varepsilon > 0$, a finite nonempty set $B \subset \mathbb R$ is called \textsl{$\varepsilon$-tight} if either $\mathrm{card}\,B = 1$ or the distance between any two consecutive elements of $B$ (in the natural order of $\mathbb R$) does not exceed $\varepsilon$.
An $\varepsilon$-tight subset $C \subseteq B$ is called \textsl{maximal} if it is not properly contained in another $\varepsilon$-tight subset of $B$. The family of all maximal $\varepsilon$-tight subsets of $B$ will be denoted by $\mathcal{M}_\varepsilon(B)$.

Every nonempty finite subset of $\mathbb R$ admits a unique decomposition into finitely many disjoint maximal $\varepsilon$-tight subsets. For such a set $B$, define
$$
\Delta_\varepsilon B\ :=\ \max\bigl\{\mathrm{diam}(C): C \in \mathcal{M}_\varepsilon(B)\bigr\}.
$$

The following gives a characterization of sequences whose achievement sets are neither finite nor Cantor sets \cite[Prop.~2.8]{MNP}.

\begin{thm}
\label{charint}
Let $\sum x_n$ be a convergent series of non-negative, non-increasing terms. Then the following are equivalent:
\begin{itemize}
\item[(i)] The achievement set $E(x_n)$ contains an interval.
\item[(ii)] $\lim\limits_{n\to\infty} \Delta_{r_n} F_n > 0$.
\item[(iii)] $\lim\limits_{k\to\infty} \Delta_{r_{n_k}} F_{n_k} > 0$ for some increasing sequence $(n_k)$ of indices.
\end{itemize}
\end{thm}

\vspace{.3in}
\section{The Lebesgue measure of the boundary}

Almost all known examples of achievable Cantorvals have boundaries of measure zero. This applies both to multigeometric examples (the Guthrie–Nymann Cantorval \cite[Thm.~5.3]{BPW}, the Ferens Cantorvals \cite[Thm.~8]{BP}, and the Guthrie–Nymann–Jones Cantorvals \cite[Thm.~3.1]{B19}) and to non-multigeometric ones (the generalized Ferens Cantorvals \cite[Thm.~3.3]{MNP} and the Marchwicki–Miska Cantorvals \cite[Thm.~9]{PP24}). The only family of achievable Cantorvals whose boundary measure remained unresolved were the Kyiv Cantorvals \cite{VMPS19}. In this note, however, we prove that their boundaries are also null sets.

The question of whether all Cantorvals, or at least all multigeometric ones, have null boundaries was first raised in 2016 by an anonymous referee of \cite{BP}, and has since been repeated several times during the Workshops on Postmodern Real Analysis in B\c{e}dlewo, Poland.

\begin{thm}
\label{lns}
The boundary of any multigeometric Cantorval  $E\,=\,E(k_1,k_2,\ldots,k_m; q)$ is a Lebesgue null set.
\end{thm}
\begin{proof}
Clearly, $\lambda(\textnormal{int}\,E)\le\lambda(E)$. Suppose that $\lambda(\textnormal{int}\,E)<\lambda(E)$. By the Lebesgue Density Theorem, there exists a set $S\subset \textnormal{Fr}\,E$ such that $\lambda(S)=\lambda(\textnormal{Fr}\,E)>0$ and, for every $x\in S$,
$$
\lim_{h\to 0+}\frac{\lambda(S^c\cap[x-h,x+h])}{2h}\ =\ 0.
$$
Since $\textnormal{int}\,E \subset S^c$, we shall contradict this equality by proving that
$$
\forall_{x\in C} \qquad \liminf_{h\to0+}\frac{\lambda(\textnormal{int}\,E\cap[x-h,x+h])}{2h}\ \ \ge\ \ \frac{\lambda(\textnormal{int}\,E)}{2qr_0}\ >\ 0.
$$
Indeed, take any $x\in C$ and $h\in(0,\,r_0]$. Let $n$ be the unique positive integer such that
\begin{equation}
\label{dow1}
h\,\in\,(r_{mn},\,r_{m(n-1)}]\ =\ (q^nr_0,\,q^{n-1}r_0].
\end{equation}
Then there exists at least one $f\in F_{nm}$ such that
\begin{equation}
\label{dow2}
[f,\,f+r_{mn}]\ \subset [x-h,\,x+h].
\end{equation}
By the self-similarity of $E$, we have
\begin{equation}
\label{dow3}
\frac{\lambda(\textnormal{int}\,E_{kn})}{r_{mn}}\ =\ \frac{\lambda(\textnormal{int}\,E)}{r_0}.
\end{equation}
Clearly, $f+\textnormal{int}\,E_{kn}\subset \textnormal{int}\,E$, and thus
\begin{align*}
\frac{\lambda(\textnormal{int}\,E\cap[x-h,x+h])}{2h}\ &\overset{\eqref{dow2}}{\ge}\ \frac{\lambda(\textnormal{int}\,E_{mn})}{r_{mn}}\cdot\frac{r_{mn}}{2h} \\[.1in]
&\overset{\eqref{dow1}, \eqref{dow3}}{\ge}\ \frac{\lambda(\textnormal{int}\,E)}{r_0}\cdot \frac{r_{mn}}{2r_{m(n-1)}}\ =\ \frac{\lambda(\textnormal{int}\,E)}{2qr_0}\ > 0.
\end{align*}
This contradiction shows that $\lambda(\textnormal{int}\,E)<\lambda(E)$ is impossible. Hence the proof is complete.
\end{proof}

The problem of whether the boundary has measure zero for every achievable Cantorval has remained open for more than eight years. Although the general answer is unknown, we now introduce a special family of achievable Cantorvals (called \textsl{standard}) for which the boundaries always have measure zero. This family contains all multigeometric Cantorvals as well as many non-multigeometric ones, thereby extending Theorem \ref{lns}.

We consider sequences $t=(t_i)_{i=1}^m$, $m\in\mathbb N\cup\{\infty\}$, with entries $0$ or $1$. They will serve as indices.  For each sequence $t=(t_i)_{i=1}^m$ we assign a value $x_t\in E$ by setting $x_t=\sum_{i=1}^mt_ix_i$. Given a finite $t=(t_i)_{i=1}^m$, a set of the form $x_t+[0,\,r_m]$ is called the $t$-brick of order $m$ and denoted by $B_t$. Although $t$ may not be unique, the order of a brick is uniquely determined by its length.

Since we are defining a special family of achievable Cantorvals $E(x_n)$, it follows from Kakeya’s classical observations (cf. \cite[Cor. 3]{PP24}) that $\textnormal{card}\,K(x_n)=\textnormal{card}\, K^c(x_n)=\infty$.

The $m$-th iteration of $E(x_n)$ is $I_m=\bigcup_{|t|=m}B_t$. We also write $I_0=[0,\,r_0]$. Clearly, $I_{m-1}=I_m$ if and only if $m\in K^c(x_n)$, while $I_{m+1}\varsubsetneq I_m$ if and only if $m\in K(x_n)$.

Let $\sum x_n$ be a convergent series of positive monotone terms such that $E=E(x_n)$ is a Cantorval. Let $P_n$ denote the longest component of $E_n$. We say that $E$ is a \textsl{standard Cantorval} if
$$\limsup\limits_{n\to \infty} \frac{|P_n|}{r_n} >0.$$

Now we can state the generalization of Theorem \ref{lns}.
\begin{thm}
\label{stan}
Let $\sum x_n$ be a convergent series of positive monotone terms such that $E=E(x_n)$ is a standard Cantorval. Then
\[
\lambda(E) = \lambda(\textnormal{int}\,E).
\]
\end{thm}
\begin{proof}
Suppose, to the contrary, that $\lambda(E) \neq \lambda(\textnormal{int}\,E)$. Then $\lambda(\textnormal{Fr}\,E) >0$. By the Lebesgue Density Theorem, there exists a measurable set $S\subset \textnormal {Fr}\,E$ with $\lambda(S)=\lambda(\textnormal{Fr}\,E)$ such that for every $x\in S$,
\begin{equation}
\label{boni}
\lim\limits_{h\to 0^+}\frac{\lambda((\textnormal{Fr}\,E)^c\cap [x-h,x+h])}{2h} = 0.
\end{equation}
Let $(n_k)$ be an increasing sequence of positive integers such that $\lim\limits_{k\to\infty}\frac{|P_{n_k}|}{ r_{n_k}}$ exists and is positive. Take any $x\in S$. Since $S\subset I_{n_k}$ for every $k$, and since each brick of order $n_k$ has length $r_{n_k}$, there exists a brick $J$ of order $n_k$ such that $x\in J\subset[x-r_{n_k},\,x+r_{n_k}]$. Hence
$$
\frac{\lambda(\textnormal{int}\,E\,\cap\,[x-r_{n_k},\,x+r_{n_k}])}{2r_{n_k}}\ \ \ge\ \ \frac{|P_{n_k}|}{2r_{n_k}}.
$$
Since $r_{n_k}\to 0$, we obtain
\begin{align*}
\limsup_{h\to 0^+}\frac{\lambda((\textnormal{Fr}\,E)^c\cap [x-h,x+h])}{2h}\ &\ge\ \limsup_{k\to\infty}\frac{\lambda(\textnormal{int}\,E\,\cap\,[x-r_{n_k},x+r_{n_k}])}{2r_{n_k}} \\
&\ge\ \frac12\lim\limits_{k\to\infty}\frac{|P_{n_k}|}{r_{n_k}}\ >\ 0,
\end{align*}
contradicting \eqref{boni}. Thus, $\lambda(E) = \lambda(\textnormal{int}\,E)$.
\end{proof}

\begin{thm}
Every multigeometric Cantorval is standard.
\end{thm}
\begin{proof}
Let $J$ be the longest component of $E=E(k_1,k_2,\ldots,k_m;\,q)$. By self-similarity of $E$, we have that $q^n J$ is the longest component of $E_{mn}$. Moreover, $r_{mn}=q^nr_0$, so
\[
\frac{|q^nJ|}{r_{mn}} = \frac{|J|}{r_0},
\]
and thus $E$ is a standard Cantorval.
\end{proof}

Actually, there are three families of achievable Cantorvals not contained in the family of multigeometric Cantorvals. It turns out that all three consist solely of standard Cantorvals, and therefore have Lebesgue measure equal to the measure of their interior. One such family is that of the \textsl{generalized Ferens Cantorvals}, introduced in \cite{MNP} to provide examples of achievable Cantorvals $A$ such that the algebraic sum of any finite number of copies of $A$ remains a Cantorval \cite[Thm. 4.17]{MNP}. This family includes all multigeometric Cantorvals studied in \cite{F}, \cite{BP}, and \cite{BGM}. For the reader’s convenience, we recall the definition.

Given a sequence $m=(m_n)$ of positive integers with $m_n\ge 2$, a sequence $k=(k_n)$ of positive integers with $k_n>m_n$ for all $n$, and a sequence $q=(q_n)$ of positive numbers, we say that a series $\sum a_j$ is a \textsl{generalized Ferens series} (abbreviated GF series) if, for $i\in\mathbb N$ and $j\in\{K_{i-1}+1, \, K_{i-1}+2,\,\ldots,\,K_i\}$, where $K_p:=\sum_{i=1}^pk_i$ for $p\in\mathbb N$ and $K_0:=0$, we have
$$
a_j\ =\ a_j(k,m,q)\ :=\ (m_i+K_i-j)q_i\ .
$$
For positive integers $p,r$ with $r>p\ge 2$, let
$$
s(p,r):=\sum_{i=1}^{r-1}(p+i).
$$
We write $s_n:=s(m_n,k_n)$. Then, for any $n\in\mathbb N$, the set of all possible subsums formed from $\{a_j:\ j\in\{K_{i-1}+1, \ldots, K_i\}\}$ (without repetition) is exactly
$$
\bigl(\{0\}\cup\{m_n,\,m_n+1,\ldots,\,s_n\}\cup\{s_n+m_n\}\bigr)q_n \qquad\text{(see \cite[Fact 3]{BPW})}.
$$

It was proved in \cite[Thm. 3.1]{MNP} that if $\sum a_j(m,k,q)$ is a convergent GF series satisfying
\begin{equation*}
\label{gf1}
\tag{GF$_1$} \ \forall_{n\in\mathbb N}\qquad q_n\ \le\ (s_{n+1}-m_{n+1}+1)q_{n+1},
\end{equation*}
and
\begin{equation*}
\label{gf2}
\tag{GF$_2$} \ \forall_{n\in\mathbb N} \qquad m_nq_n\ >\ \sum_{i>n}(s_i+m_i)q_i,
\end{equation*}
then $E(a_j)$ is a Cantorval. These Cantorvals are called \textsl{generalized Ferens Cantorvals}. We now show that all of them are standard. For this we need, besides the proof of \cite[Thm. 3.1]{MNP}, the following simple fact.

\begin{lem}
\label{delH}
If $\sum a_i$ and $\sum b_i$ are two convergent series of positive terms, then for every $k\in\mathbb N$,
$$
\inf_{i>k}\,\frac{a_i}{b_i}\ \le\ \ \frac{\sum_{i>k}a_i}{\sum_{i>k}b_i}\ \ \le\ \sup_{i>k}\,\frac{a_i}{b_i}.
$$
\end{lem}
\begin{proof}
This follows directly from
$$
\frac{\sum_{i>k}a_i}{\sum_{i>k}b_i}\ = \ \sum_{i>k}\,\frac{b_i}{\sum_{j>k}b_j}\cdot\frac{a_i}{b_i},
$$
since the coefficients $\frac{b_i}{\sum_{j>k}b_j}$ are nonnegative and sum to 1.
\end{proof}

From the proof of \cite[Thm. 3.1]{MNP}, it follows that $E(a_j)$ contains an interval $P$ of length at least $\sum_{i=1}^\infty(s_i-m_i)q_i$. Define, for $n\in\mathbb N$, $q^{(n)}:=(q_i)_{i=n+1}^\infty$, $k^{(n)}:=(k_i)_{i=n+1}^\infty$, $m^{(n)}:=(m_i)_{i=n+1}^\infty$. Then for each $n$, the series $\sum_{j=1}^\infty a_j(k^{(n)},m^{(n)},q^{(n)})$ is a generalized Ferens series satisfying both (GF$_1$) and (GF$_2$). Its achievement set contains an interval $P^{(n)}$ of length at least
$$
\sum_{i=1}^\infty\bigl(s_i^{(n)}-m_i^{(n)}\bigr)q_i^{(n)}\ =\ \sum_{i>n}(s_i-m_i)q_i.
$$
Moreover,
$$
E\bigl(a_j(k^{(n)},m^{(n)},q^{(n)})\bigr)\ =\ E_{K_n}\bigl(a_j(k,m,q)\bigr).
$$
Hence, for $\sum a_j(k,m,q)$ we obtain
\begin{align*}
\limsup\limits_{n\to\infty}\frac{|P_n|}{r_n}\ &\ge\ \limsup\limits_{n\to\infty}\frac{|P^{(n)}|}{r_{K_n}}\\
&=\ \limsup\limits_{n\to\infty}\frac{\sum_{i>n}(s_i-m_i)q_i}{\sum_{i>n}(s_i+m_i)q_i}\\[.1in]
&\overset{\text{Lemma \ref{delH}}}{\ge}\ \limsup\limits_{n\to\infty}\,\inf_{i>n}\frac{s_i-m_i}{s_i+m_i}\\
&=\ \liminf\limits_{n\to\infty}\left(1-\frac{2m_i}{s_i+m_i}\right).
\end{align*}
Since $s_i\ge\frac32m_i(m_i+1)$ \cite[Fact 3]{BPW} and $m_i\ge 2$, we conclude that $\limsup\limits_{n\to\infty}\frac{|P_n|}{r_n}\ge\frac7{11}$. This completes the proof of the following theorem.

\begin{thm}
\label{GFC}
All generalized Ferens Cantorvals satisfying conditions (GF$_1$) and (GF$_2$) are standard Cantorvals.
\end{thm}

As an immediate consequence, Theorem \ref{stan} implies that the Lebesgue measure of the boundary of any such Cantorval is zero. This fact was originally established by direct computation in \cite[Thm. 3.3]{MNP}.

\medskip

Another family of non-multigeometric Cantorvals was introduced in \cite{PP24}, showing that $E(a_n)$ can be a Cantorval even when the set of indices $n$ with $a_n\le r_n$ has asymptotic density zero \cite[Cor. 8]{PP24}. These are the \textsl{Marchwicki--Miska Cantorvals}, defined as follows.

For a positive integer $n$, define
$$
b_1^n:=2^{n+1},\quad b_2^n:=2^n+1,\quad b_i^n:=2^{n+3-i}\ \ \text{for $3\le i\le n+2$.}
$$
The set of all subsums of $b_i^n$, $i=1,\,2,\,\ldots,\,n+2$, is exactly
\begin{multline*}
\left\{2i-2: \ i\in\{1,2,\ldots,2^{n-1}\}\right\}\\
\cup\ \left\{i:\ 2^n\le i\le 4\cdot2^n-1\right\}\\
\cup\ \left\{4\cdot2^n-1+2i:\ i\in\{1,2,\ldots,2^{n-1}\}\right\}.
\end{multline*}

Now take any increasing sequence $N=(N_s)_{s=1}^\infty$ of positive integers with $N_s-N_{s-1}\ge 3$ for all $s\in\mathbb N$ (with $N_0:=0$). Define $n_s:=N_s-N_{s-1}-2$ for all $s$, and
$$
q_1:=1,\qquad q_{k+1}:=\frac1{3\cdot2^{n_{k+1}}}q_k \quad (k\in\mathbb N).
$$
We then define the sequence $a_j=a_j(N)$ by setting $a_{N_{k-1}+i}:=b_i^{n_k}q_k$ for $k\in\mathbb N$ and $i\in\{1,2,\ldots,n_k+2\}$. Thus, $(a_j)$ consists of groups of terms: the $k$-th group corresponds to $n=n_k$, scaled by $q_k$, with indices from $N_{k-1}+1$ to $N_k$, containing $n_k+2$ terms. In particular,
$$
r_{N_k}=\sum_{j=k}^\infty(5\cdot2^{n_{j+1}}-1)q_{j+1}\qquad\text{for all $k$.}
$$
It was shown in the proof of \cite[Thm. 7]{PP24} that $E(a_j)=E(a_j(N))$ is a Cantorval containing a component interval of length $\sum_{i=1}^\infty(3\cdot2^{n_i}-1)q_i$.

Define $N^{(s)}:=(N_k)_{k=s+1}^\infty$. Since $E_{N_s}(a_j)=q_{s+1}E(a_j(N^{(s)}))$ for all $s$, and $E(a_j(N^{(s)}))$ is also a Marchwicki--Miska Cantorval, we see that $E_{N_s}(a_j)$ contains an interval $P^{(s)}$ of length $\sum_{i=s+1}^\infty(3\cdot2^{n_i}-1)q_i$. The estimate proceeds as in the Ferens case:
\begin{align*}
\limsup\limits_{s\to\infty}\frac{|P_s|}{r_s}\ &\ge\ \limsup\limits_{s\to\infty}\frac{|P_{N_s}|}{r_{N_s}}\\
&\ge\ \limsup\limits_{s\to\infty}\frac{|P^{(s)}|}{r_{N_s}}\\
&\ge\ \limsup\limits_{s\to\infty}\frac{\sum_{i=s+1}^\infty(3\cdot2^{n_i}-1)q_i}{\sum_{i=s+1}^\infty(5\cdot2^{n_i}-1)q_i}\\[.1in]
&\ge\ \limsup\limits_{s\to\infty}\inf_{i>s}\frac{3\cdot2^{n_i}-1}{5\cdot2^{n_i}-1}\ \ge\ \frac59.
\end{align*}

\begin{thm}
\label{MMstan}
All Marchwicki--Miska Cantorvals are standard.
\end{thm}

By Theorem \ref{stan}, their boundaries have Lebesgue measure zero, a fact previously computed directly in \cite[Thm. 9]{PP24}.

\medskip

The third large family of non-multigeometric achievable Cantorvals are the \textsl{Kyiv Cantorvals}. In this case, however, the Lebesgue measure of their boundaries has resisted direct computation, and therefore we devote the next section to them.

\vspace{.3in}
\section{The Kyiv Cantorvals}

The basic idea of the following construction is taken directly from \cite{VMPS19}, but our assumptions on the sequences $(s_k)$ and $(m_k)$ differ slightly, since we base our presentation on the concept of $\varepsilon$-tight subsets introduced in \cite{MNP}. We do not analyze the precise relationship between the two sets of assumptions, as no such comparison is needed here.

Let $(m_k)$ and $(s_k)$ be sequences of positive integers such that
\begin{gather}
s_n \ \ge \ 3m_n-4 \qquad \text{for all $n$,} \label{ass1}\\
m_n\ \ge \ 3 \qquad \text{for all $n$,} \label{ass2}\\
\limsup\limits_{n\to\infty}m_n\ \ge\ 4.\label{ass3}
\end{gather}

We now define a series $\sum d_i$ inductively by grouping consecutive terms and introducing a sequence  $(a_k)$ of principal values associated with each group. The $k$-th group of terms is indexed by $N_{k-1}+1,\,N_{k-1}+2,\,\ldots,\,N_k$. We require that the first $s_k+1$ terms of the $k$-th group equal $a_k$, while the remaining $m_k$ terms equal $\frac{m_k-1}{m_k}a_k$. In particular, $N_k-N_{k-1}=s_k+m_k+1$, and we set $N_0:=0$. The sum of all terms in the $k$-th group is
\begin{equation}
\label{alfA}
G_k\ :=\ \sum_{i=N_{k-1}+1}^{N_k}d_i\ =\ (s_k+m_k)a_k.
\end{equation}

The principal values $a_k$ are defined so that $\sum d_i=\sum_j G_j=1$ and
\begin{equation}
\label{betA}
\forall_{k\in\mathbb N}\qquad r_{N_k} \ =\ \sum_{i>N_k}d_i\ :=\ \frac{2a_k}{m_k}.
\end{equation}
In particular,
$$
1 = \sum_{i=1}^\infty d_i = \sum_{i=1}^{N_1}d_i + r_{N_1} = (s_1+m_1)a_1 + \frac{2a_1}{m_1},
$$
hence
$$
a_1 = \frac{m_1}{m_1^2+s_1m_1+2}.
$$
By induction, one shows easily that
$$
\forall_{n\in \N} \qquad a_n = \frac{2^{n-1}m_n}{\prod_{k=1}^n(m_k^2+s_km_k+2)}.
$$
Indeed,
$$
\frac{2a_k}{m_k} \,\overset{\eqref{betA}}{=}\, r_{N_k} \overset{\eqref{alfA}}{=}\, (s_{k+1}+m_{k+1})a_{k+1}+r_{N_{k+1}} \overset{\eqref{betA}}{=}\, (s_{k+1}+m_{k+1})a_{k+1}+\frac{2a_{k+1}}{m_{k+1}},
$$
which yields
$$
a_{k+1} = \frac{2m_{k+1}}{m_k(m_{k+1}^2+s_{k+1}m_{k+1}+2)}\, a_k.
$$

As a corollary \cite[Cor.~2.2]{VMPS19}, for all $k\in\mathbb N$ we obtain:
$$
r_{N_k} = \frac{2^k}{\prod_{i=1}^k(m_i^2+s_im_i+2)},\qquad
\frac{r_{N_{k+1}}}{r_{N_k}} = \frac2{m_{k+1}^2+s_{k+1}m_{k+1}+2},
$$
and
\begin{equation}
\label{doublestar}
\frac{a_{k+1}}{a_k} = \frac{2m_{k+1}}{m_k(m_{k+1}^2+s_{k+1}m_{k+1}+2)}.
\end{equation}

Now consider the set
$$
S_k := \{\sum_{i\in A} d_i:\ A \subset \{N_{k-1}+1,\ldots, N_k\}\},
$$
which contains
$$
D_k := \left\{\frac{a_k}{m_k}i:\ i\in\mathbb N,\ (m_k-3)m_k+2 \ \le\ i\ \le\ (s_k+3)m_k-2\right\}.
$$
Hence,
$$
\Delta_{\frac{a_k}{m_k}}S_k \ \ge\ (s_k-m_k+6)a_k - \frac{4a_k}{m_k}.
$$

We define large $\tfrac{a_k}{m_k}$-tight subsets of $F_{N_k}(d_i)$ by induction:
\begin{align*}
C_1 &:= D_1, \\
C_{n+1} &:= C_n + D_{n+1}.
\end{align*}
With this definition,
$$
\min C_n = \sum_{k=1}^n\left(m_k-3+\frac2{m_k}\right)a_k, \qquad
\max C_n = \sum_{k=1}^n\left(s_k+3-\frac2{m_k}\right)a_k.
$$

The proof of $\tfrac{a_k}{m_k}$-tightness of $C_k$ proceeds by induction. It reduces to showing that if $\alpha<\beta$ are two consecutive points of $C_k$, then
$$
\alpha + \max D_{k+1} + \frac{a_{k+1}}{m_{k+1}} \ \ge\ \beta + \min D_{k+1}.
$$
Since $\beta-\alpha \le \tfrac{a_k}{m_k}$ by the induction hypothesis, it suffices to prove
$$
\bigl(s_{k+1}+3-\tfrac1{m_{k+1}} - m_{k+1}+3-\tfrac2{m_{k+1}}\bigr)a_{k+1}\ \ge\ \frac{a_k}{m_k}.
$$
In light of \eqref{doublestar}, this is equivalent to
$$
\frac2{s_{k+1}+m_{k+1}+\tfrac2{m_{k+1}}}\ \ge\ \frac1{s_{k+1}-m_{k+1}+6-\tfrac3{m_{k+1}}},
$$
that is,
\begin{equation}
\label{estrella}
m_{k+1}(s_{k+1}-3m_{k+1}+12) - 8 \ \ge\ 0,
\end{equation}
which follows immediately from \eqref{ass1} and \eqref{ass2}.

This completes the proof that every set $C_k$ is $\tfrac{a_k}{m_k}$-tight, hence $r_{N_k}$-tight as well. Therefore,
$$
\Delta_{r_{N_k}}F_{N_k}(d_i) \ \ge\ \mathrm{diam}\,C_k,
$$
which implies
$$
\lim_{k\to\infty}\Delta_{r_{N_k}}F_{N_k}(d_i) \ \ge\ \lim_{k\to\infty}\mathrm{diam}\,C_k = \sum_{n=1}^\infty\left(s_n-m_n+6-\frac4{m_n}\right)a_n > 0.
$$
Hence $E(d_i)$ contains an interval by \cite[Prop.~2.8]{MNP}. Moreover, if $m_k>3$, then $d_{N_k}>r_{N_k}$, and by assumption \eqref{ass3}, $E(d_i)$ has infinitely many gaps. Consequently, by the Guthrie–Nymann Classification Theorem, $E(d_i)$ is a Cantorval.

\begin{thm}
\label{Kyiv}
Every Kyiv Cantorval is a standard Cantorval, and its measure equals the measure of its interior.
\end{thm}

\begin{proof}
By Theorem~\ref{stan}, it suffices to show that all Kyiv Cantorvals are standard. Let $(d_i)_{i=1}^\infty = \bigl(d_i((s_n),(m_n))\bigr)_{i=1}^\infty$ be a sequence generating a Kyiv Cantorval. For all $k$ with $m_k \ge 4$, we have
$$
E_{N_k}(d_i) = \frac{2a_k}{m_k}E(d_i^{(k)}),
$$
where $(d_i^{(k)})_{i=1}^\infty := \bigl(d_i((s_n)_{n>k},(m_n)_{n>k})\bigr)_{i=1}^\infty$. Thus, by our earlier considerations, $E_{N_k}(d_i)$ contains an interval $P^{(k)}$ of length
$$
|P^{(k)}| = \sum_{n>k}\left(s_n-m_n+6-\frac4{m_n}\right)a_n.
$$
Therefore,
\begin{align*}
\limsup_{s\to\infty}\frac{|P_s|}{r_s}
&\ge \limsup_{k\to\infty} \frac{|P^{(k)}|}{r_{N_k}}
= \limsup_{k\to\infty} \frac{m_k\sum_{n>k}\left(s_n-m_n+6-\tfrac4{m_n}\right)a_n}{2a_k} \\[.1in]
&\ge \limsup_{k\to\infty}\frac{ m_k}2\left(s_{k+1}-m_{k+1}+6-\tfrac4{m_{k+1}}\right)\frac{a_{k+1}}{a_k}\\[.1in]
&\overset{\eqref{doublestar}}{\ge} \limsup_{k\to\infty} \frac{s_{k+1}-m_{k+1}+6-\tfrac4{m_{k+1}}}{s_{k+1}+m_{k+1}+\tfrac2{m_{k+1}}}\overset{\eqref{ass3}}{=} \limsup\limits_{\substack{k\to\infty \\ m_k\ge4}} \frac{s_{k}-m_{k}+6-\tfrac4{m_{k}}}{s_{k}+m_{k}+\tfrac2{m_{k}}} \\[.1in]
&\ge \limsup\limits_{\substack{k\to\infty \\ m_k\ge4}} \frac{s_k-m_k+5}{s_k+m_k+1}
= \limsup\limits_{\substack{k\to\infty \\ m_k\ge4}} \left(1-\frac{2m_k-4}{s_k+m_k+1}\right)\\[.1in]
 &\overset{\eqref{ass1}}{\ge}\ \limsup\limits_{\substack{k\to\infty \\ m_k\ge4}}\left(1-\frac{2m_k-4}{4m_k-3}\right)\ \ge\ \frac12,
\end{align*}
which completes the proof.
\end{proof}

\begin{con}
Every achievable Cantorval is a standard Cantorval, and hence its measure equals the measure of its interior.
\end{con}

\vspace{.3in}
\section{Remarks on the set of uniqueness}

In this section we discuss the topological properties of the set of subsums with unique representations and slightly extend the results of the primary paper on this topic \cite{GM23}.
For an absolutely convergent series $\sum x_n$, its set of subsums with unique representation $U=U(x_n)$ is defined by
$$
U\ =\ U(x_n)\ :=\ \left\{x\in E(x_n):\ \exists!_{ A\subset \mathbb N}\quad x=\sum_{n\in A}x_n\,\right\}.
$$
The complement
$$
U^c=\Bigl\{x\in E:\ \exists_{A,B\subset\mathbb N},\ A\ne B,\ \ x=\sum_{n\in A}x_n=\sum_{n\in B}x_n\Bigr\}
$$
is the set of subsums of $\sum x_n$ with multiple representations.
Since $E(x_n)=\sum_{n:\,x_n<0}x_n+E(|x_n|)$, it is enough to consider convergent series with positive, monotone terms.

\begin{obs}
\label{alien}
$\sum_{n\in A}x_n\in U$ if and only if $E((x_n)_{n\in A})\ \cap\ E((x_n)_{n\in\mathbb N\setminus A})\ =\ \{0\}$.
\end{obs}

\begin{proof}
($\Rightarrow$) If $\sum_{n\in A}x_n\in U$, then in particular $\sum_{n\in A}x_n\ne \sum_{n\in A^c}x_n$. Suppose there exists a positive $x\in E((x_n)_{n\in A})\cap E((x_n)_{n\in A^c})$. Then there are non-empty subsets $I_1\subset A$ and $I_2\subset A^c$ such that $x=\sum_{n\in I_1}x_n=\sum_{n\in I_2}x_n$, which implies
$$
\sum_{n\in A}x_n = \sum_{n\in I_2\cup(A\setminus I_1)}x_n,
$$
contradicting uniqueness.

($\Leftarrow$) Conversely, suppose that $E((x_n)_{n\in A})\ \cap\ E((x_n)_{n\in\mathbb N\setminus A})\ =\ \{0\}$ and $\sum_{n\in A}x_n=\sum_{n\in B}x_n$ for some $B\ne A$. Then
$$
\sum_{n\in A\setminus B}x_n = \sum_{n\in B\setminus A}x_n,
$$
which contradicts the assumption, since $B\setminus A$ is a non-empty subset of $A^c$.
\end{proof}

In particular, if $r_k<x_k$, then the sum $r_k$ is uniquely represented.

\begin{obs}
\label{obs2}
$x\in U^c$ if and only if there exists $k\in\mathbb N$ and distinct $f,g\in F_k$ such that
$$
x\in(f+E_k)\cap(g+E_k).
$$
\end{obs}

We omit the elementary proof of the observation.

By $A^d$ we denote the set of all accumulation points of the set $A$.
\begin{lem}
\label{fourcond}
The following are equivalent:
\begin{itemize}
\item[(i)] $\overline{U^c}=E$;
\item[(ii)] $\mathrm{int}_E U=\emptyset$;
\item[(iii)] $0\in (U^c)^d$;
\item[(iv)] $\forall_{k\in\mathbb N},\ E_k\cap U^c\ne\emptyset$.
\end{itemize}
\end{lem}

\begin{proof}
(i)$\Leftrightarrow$(ii) is immediate from topology.

(ii)$\Rightarrow$(iii): If $E$ is finite then $0\in\mathrm{int}_E U$, contradicting (ii). Thus $(x_n)$ must be infinite, and $E$ is a non-empty perfect set. Hence every neighbourhood of $0$ in $E$ contains uncountably many points, at least one of which must lie in $U^c$. Thus, $0\in (U^c)^d$.

(iii)$\Rightarrow$(iv): Given $k$, by (iii) there exists $x\in U^c\cap(0,x_k)$. Since $x<x_k$, we must have $x\in E_k$.

(iv)$\Rightarrow$(ii): Suppose instead that there exists an interval $(\alpha,\beta)$ such that  $\emptyset\ne(\alpha,\beta)\cap E\subset U$. Since $\bigcup_kF_k$ is dense in $E$, we can find $f\in E\cap(\alpha,\beta)$ with the unique and finite representation $f=\sum_{n\in C}x_n$. Let $m=\max C$, and choose $k$ with $r_k<\min\{x_m,\beta-f\}$. Then $f+E_k \subset (\alpha, \beta)\cap E \subset U $, so also $E_k \subset U$, and so $E_k\cap U^c=\emptyset$, contradicting (iv).
\end{proof}

\begin{lem}
\label{sierp}
 Let $\{Z_i\}$ be a countable family of closed proper subsets of $[a,b]$ such that $\bigcup_i Z_i=[a,b]$. Then there exist $i\ne j$ with $Z_i\ne Z_j$ and
$$
Z_i\cap Z_j\cap (a,b)\ne\emptyset.
$$
\end{lem}

\begin{proof}
We consider only the case where the family $\{Z_i\}$ is infinite and $Z_i\ne Z_j$ whenever $i\ne j$. Suppose the claim does not hold, that is, no two elements of $\{Z_i\}$ intersect in the interior of $[a,b]$. Without loss of generality, we may assume that the family is indexed by the positive integers, $\{Z_i\}=\{Z_i:\ i\in\mathbb N\}$.

Let $l_1$ be the smallest index such that $Z_{l_1}\cap(a,b)\ne\emptyset$, and let $k_1$ be the smallest index greater than $l_1$ with $Z_{k_1}\cap(a,b)\ne\emptyset$. Then there exists an open interval $(\alpha_1,\beta_1)$ such that, for $D_1:=\bigcup_{i=1}^{k_1}Z_i$, we have $\alpha_1,\beta_1\in D_1$, $(\alpha_1,\beta_1)\cap D_1=\emptyset$, and $[\alpha_1,\beta_1]\subset(a,b)$.

Next, let $l_2:=\min\{i>k_1:\ Z_i\cap(\alpha_1,\beta_1)\ne\emptyset\}$ and $k_2:=\min\{i>l_2:\ Z_i\cap(\alpha_1,\beta_1)\ne\emptyset\}$. Define $D_2:=\bigcup_{i=1}^{k_2}Z_i$. Since the sets $Z_i$ are closed and do not intersect each other in the interior of $[a,b]$, there exist $\alpha_2,\beta_2\in D_2\setminus D_1$ such that $(\alpha_2,\beta_2)\cap D_2=\emptyset$ and $[\alpha_2,\beta_2]\subset(\alpha_1,\beta_1)$.

Proceeding inductively, we construct an increasing sequence $(k_n)$ of positive integers and a decreasing sequence $\bigl((\alpha_n,\beta_n)\bigr)_{n\in\mathbb N}$ of open intervals such that $[\alpha_{n+1},\beta_{n+1}]\subset(\alpha_n,\beta_n)$ and $(\alpha_n,\beta_n)\cap\bigcup_{i=1}^{k_n}Z_i=\emptyset$ for all $n$. The set $\bigcap_n[\alpha_n,\beta_n]$ is then a non-empty subset of $[a,b]\setminus\bigcup_iZ_i$, contradicting the assumption that $\bigcup_iZ_i=[a,b]$.

\end{proof}

\begin{cor}
\label{cor5}
If $E$ contains an interval, then $U^c$ is dense in $E$.
\end{cor}

\begin{proof}
Suppose $[\alpha,\beta]\subset E$ and choose $k$ with $r_k<\beta-\alpha$. Applying Lemma \ref{sierp} to $\{(f+E_k)\cap[\alpha,\beta]: f\in F_k\}$, Observation~\ref{obs2} yields a point $x\in E\cap U^c$. Thus, every open subinterval of $[\alpha,\beta]$ meets $U^c$, so
 $U^c\cap \mathrm{int}\,E$ is dense in $\mathrm{int}\,E$.
Since $E$ is regularly closed, $E=\overline{\mathrm{int}\,E}\subset\overline{U^c}\subset E$ and the conclusion follows.
\end{proof}

\begin{cor}
\label{cor6}
$\mathrm{int}\,U=\emptyset$ always.
\end{cor}

\begin{proof}
By the Guthrie–Nymann Classification Theorem, $E$ is either nowhere dense or contains an interval. In the first case, $\mathrm{int}\,U\subset \mathrm{int}\,E=\emptyset$. In the second case, Corollary~\ref{cor5} and Lemma~\ref{fourcond} yield $\mathrm{int}_E U=\emptyset$.
\end{proof}

\begin{thm}
$U$ is always a $\mathcal{G}_\delta$ set.
\end{thm}

\begin{proof}
Since $E$ is closed, it suffices to show $U^c$ is $\mathcal{F}_\sigma$. From Observation~\ref{obs2},
$$
U^c=\bigcup_{k=1}^\infty S_k, \qquad S_k:=\bigcup_{\substack{f,g\in F_k \\ f\ne g}}(f+E_k)\cap(g+E_k).
$$
Each $S_k$ is closed, and $(S_k)$ is increasing, so $U^c$ is $\mathcal{F}_\sigma$.
\end{proof}

\begin{cor}
\label{cor8}
If $U$ is dense in $E$, then $U^c$ is meager in $E$.
\end{cor}

Now, we show a new sufficient condition for $E(x_n)$ to be a Cantor set, extending the well-known fact that if $E(x_n)=U(x_n)$, then $E(x_n)$ is Cantor (cf.~\cite[p.~3]{MM}).
Let $y_i>0$, $K_i\in\mathbb N$. The sequence
$$
(y_i; K_i):=(\underbrace{y_1,\ldots,y_1}_{K_1\ \text{times}},\,\underbrace{y_2,\ldots,y_2}_{K_2\ \text{times}},\,\ldots)
$$
denotes repetition of each $y_i$ exactly $K_i$ times.

\begin{thm}
\label{warCan}
Let $(x_n)=(y_i;K_i)$ with $\sum x_n<\infty$. If for each $x\in E(x_n)$ there exists exactly one sequence $(n_i)\in S:=\prod_{j=1}^\infty\{0,1,\ldots,K_j\}$ such that
$$
x=\sum_{i=1}^\infty n_i y_i,
$$
then $E(x_n)$ is a Cantor set.
\end{thm}

\begin{proof}
With the product topology, $S$ is homeomorphic to $\{0,1\}^\mathbb N$, and hence to the Cantor set. The map $T:E(x_n)\to S$ defined by $T(x)=(n_i)$, where $x=\sum n_i y_i$, is a continuous bijection onto $S$. Since $E(x_n)$ is compact, $T$ is a homeomorphism, so $E(x_n)$ is Cantor.
\end{proof}

This theorem provides a quick proof of a known fact  \cite[Thm.~16]{BFPW2}.
\begin{cor}
\label{semifast}
The achievement set of any semi-fast convergent series is a Cantor set.
\end{cor}

\begin{proof}
Let $\sum x_n=\sum(y_i;K_i)$ be semi-fast convergent. Suppose $\sum n_i y_i=\sum m_i y_i$ with $(n_i)\ne(m_i)$. Let $k$ be the first index with $n_k\ne m_k$, assume $n_k>m_k$. Since the series is semi-fast convergent, $y_k>\sum_{i>k}K_i y_i$. Thus,
$$
\sum m_i y_i \ <\ \sum_{i=1}^{k-1} m_i y_i + (m_k+1)y_k \ \le\ \sum n_i y_i,
$$
a contradiction. Hence each $x\in E(x_n)$ has a unique representation, and Theorem~\ref{warCan} applies.
\end{proof}

\begin{obs}
\label{lastcor}
If $x_n< r_n$ only for finitely many $n$, then $E(x_n)$ is either a multi-interval set or a Cantor set.
More precisely, if $x_n=r_n$ for all sufficiently large $n$, then $E(x_n)$ is a multi-interval set; otherwise, $E(x_n)$ is a Cantor set.
\end{obs}

\begin{proof}
Since the topological types of $E$ and any of its remainders $E_k$ are the same, we may assume that $x_n\ge r_n$ for all $n$.
By Kakeya’s Theorem \cite[Cor. 21.14]{BFPW1}, $E(x_n)$ is a multi-interval set if $x_n=r_n$ eventually. Otherwise, let $(n_i)$ be the sequence of all indices with $x_{n_i}>r_{n_i}$. Clearly, $I_1=[0,\,\sum x_n]$ if $1\not\in (n_i)$, and $I_1=[0,\,r_1]\sqcup[x_1,\,\sum x_n]$ if $1=n_1$. The equality $I_n=I_{n+1}$ holds if and only if $n+1=n_i$ for some $i$. In general, $I_n$ is a union of $2^n$  non-overlapping bricks of order n, although there can be many pairs of bricks with a common endpoint. If $n+1=n_i$ for some $i$, then for each brick $B_t=[x_t,\,x_t+r_n]$ of order $n$ we have
$$
I_{n+1}\cap B_t\ =\ [x_t, x_t+r_{n+1}]\ \sqcup [x_t+x_{n+1},\,x_t+r_n].
$$
Thus, if $p$ was a common endpoint of two bricks of order $n$, then the interval $[p-r_{n+1},\,p+r_{n+1}]$ is the component interval of $I_{n+1}$ containing $p$. If an endpoint $p$ of a brick of order $n$ was not shared with another brick of order $n$, then the component interval of $I_{n+1}$ cointaing $p$ is either $[p-r_{n+1},\,p]$ or $[p,\,p+r_{n+1}]$. We conclude  that the maximal length of component intervals of an iteration $I_{n_i}$ is equal to $2r_{n_i}$ if there is an index $j<n_i$ with $x_j=r_j$, or is equal to $r_{n_i}$ otherwise. Thus, $E(x_n)=\bigcap_{i=1}^\infty I_{n_i}$ does not contain any interval and hence, by the Guthrie-Nymann Classification Theorem $E(x_n)$ is a Cantor set.
\end{proof}

It follows that if $x_n= r_n$ for all sufficiently large $n$ , then $E(x_n)$ is a multi-interval set such that  for any subsequence $(x_{n_k})$ the corresponding achievement set $E(x_{n_k})$ is either again a multi-interval set or a Cantor set. We believe that the duality of topological types of achievement set generated by subsequences in the corollary is exceptional, but we have no proof of it. Hence, we conclude this note with the following open problem.
\begin{OP}
Let $E(x_n)$ be a multi-interval set and $x_n<r_n$ for infinitely many $n$. Does there exist a subsequence $(x_{n_k})$ such that $E(x_{n_k})$ is a Cantorval ?
\end{OP}

\bibliographystyle{amsplain}

\begin{thebibliography}{10}
\bibitem{BBFS}
T. Banakh, A. Bartoszewicz, M. Filipczak, E. Szymonik, \textit{Topological and measure properties of some self-similar sets}, Topol. Methods Nonlinear Anal., \textbf{46}(2) (2015), 1013--1028
\bibitem{B19} M. Banakiewicz, \textit{The Lebesgue measure of some M-Cantorval}, J. Math. Anal. Appl. \textbf{471} (2019) 170--179
\bibitem{BP} M. Banakiewicz, F. Prus-Wi\'sniowski, \textit{M-Cantorvals of Ferens type}, Math. Slovaca, \textbf{67}(4) (2017), 1--12
\bibitem{BFPW1} A. Bartoszewicz, M. Filipczak, F. Prus-Wi\'sniowski, \textit{Topological and algebraic aspects of subsums of series}, Traditional and present-day topics in real analysis, 345--366, Faculty of Mathematics and Computer Science. University of \L \'od\'z, \L \'od\'z, 2013.
\bibitem{BFPW2} A. Bartoszewicz, M. Filipczak, F. Prus-Wi\'{s}niowski, \textit{Semi-fast convergent sequences and $k$-sums of central Cantor sets}, European J. Math. \textbf{6} (2020), 1523--1536
\bibitem{BFS}
A. Bartoszewicz, M. Filipczak, E. Szymonik, \textit{Multigeometric sequences and Cantorvals}, Cent. Eur. J. Math. 12(7) (2014), 1000--1007.
\bibitem{BGM} A. Bartoszewicz, Sz. G\l\c{a}b, J. Marchwicki, \textit{Recovering a purely atomic finite measure from its range}, J. Math. Anal. Appl. \textbf{467} (2018) 825--841.
\bibitem{BPW}  W. Bielas, S. Plewik, M. Walczy\'nska, \textit{On the center of distances}, Eur. J. Math. \textbf{4} (2018), no. 2, 687--698.
\bibitem{EG}
P. Erd\"{o}s, R.L. Graham, \textit{Old and new problems and results in combinatorial number theory},
volume 28 of Monographies de L’Enseignement Math´ematique. Universit´e de Gen`eve, L’Enseignement
Math´ematique, Geneva, 1980.
\bibitem{F} C. Ferens, \textit{On the range of purely atomic probability
measures}, Studia Math., \textbf{77}(3) (1984), 261--263.
\bibitem{GM23} Sz. G\l\c{a}b, J. Marchwicki, \textit{Set of uniqueness for Cantorvals},  Results Math. \textbf{78}, 9 (2023). https://doi.org/10.1007/s00025-022-01777-3
\bibitem{GN88}
 J.A. Guthrie, J.E. Nymann, \textit{The topological structure of the set of subsums of an infinite series}, Colloq. Math. \textbf{55} (1988), 323--327
\bibitem{H} Hornig, H., \textit{ \:{U}ber beliebige Teilsummen absolut konvergenter Reihen}, Studia Math. Phys. \textbf{49} (1941), 316--320
\bibitem{Jones} R. Jones, \textit{Achievement sets of sequences}, Am. Math. Mon. \textbf{118}(6) (2011), 508--521.
\bibitem{Kakeya} S. Kakeya, \textit{On the partial sums of an infinite series}, T\^{o}hoku Sc. Rep. \textbf{3} (1915), 159--164.
\bibitem{Kakeya2} S. Kakeya, \textit{On the set of partial sums of an infinite series}, Proc. Tokyo Math.-Phys. Soc., 2nd series, \textbf{7} (1914), 250--251,  \text{\path{ https://doi.org/10.11429/ptmps1907.7.14_250}}
    \bibitem{Kov25}
Kovač, V., \textit{On the Set of Points Represented by Harmonic Subseries}, Amer. Math. Monthly (2025), 1--17. https://doi.org/10.1080/00029890.2025.2540753
\bibitem{KT}
V. Kova\v{c}, T. Tao, \textit{On several irrationality problems for Ahmes series}, Acta Math. Hungar. \textbf{175}(2) (2025), 572--608. https://doi.org/10.1007/s10474-025-01528-0
\bibitem{KN24}
M. Kula, P. Nowakowski, \textit{Achievement sets of series in $\mathbb R^2$}, Results Math. \textbf{79} (2024), art.no. 221, https://doi.org/10.1007/s00025-024-02239-8
\bibitem{MM}
J. Marchwicki, J. Miska, \textit{On Kakeya conditions for achievement sets}, Result in Math. \textbf{76}(2021), art. no. 181; //doi.org./10.1007/s00025-021-01479-2
\bibitem{MNP} J. Marchwicki, P. Nowakowski, F. Prus-Wiśniowski, \textit{Algebraic sums of achievable sets involving Cantorvals}, arXiv: 2309.01589v1[math. CA], 4 Sep 2023.
\bibitem{MO} P. Mendes, F. Oliveira, \textit{On the topological structure of the arithmetic sum of two Cantor sets}, Nonlinearity \textbf{7} (1994), 329--343
\bibitem{NS} J.E. Nymann, R.A. S\'{a}enz, \textit{On the paper of Guthrie and Nymann on subsums of infinite series}, Colloq. Math. \textbf{83} (2000), 1--4.
\bibitem{PP24} F. Prus-Wiśniowski, J. Ptak, \textit{Achievable Cantorvals almost without reversed Kakeya conditions},  arXiv:2412.08768 [math CA] (2024)
\bibitem{VMPS19}
Vinishin Y., Markitan V., Pratsiovytyi M., Savchenko I., \textit{Positive series, whose sets of subsums are Cantorvals}, Proc. International Geom. Center \textbf{12}(2) (2019), 26--42. (in Ukrainian);
https://doi.org/10.15673/tmgc.v12i2.1455
\bibitem{WS} A.D. Weinstein, B.E. Shapiro, \textit{On the structure of a set of $\overline{\alpha}$-representable numbers}, Izv. Vys\v{s}. U\v{c}ebn. Zaved. Matematika. \textbf{24} (1980), 8--11.
\end{thebibliography}

\end{document}